\documentclass[11pt,reqno]{article}
\usepackage{amsfonts}
\usepackage{amsmath}
\usepackage{amssymb}
\usepackage{epsfig}
\usepackage{amsthm}
\usepackage[english]{babel}
\usepackage{hyperref}
\newcommand{\N}{\mathbb N}

\newcommand{\R}{\mathbb R}

\newcommand{\E}{\mathbb E}
\newcommand{\pr}{\mathbb P}

\newcommand{\half}{\frac{1}{2}}
\newcommand{\salj}{\mathcal{F}}

\newcommand{\Var}{\mathop{\rm Var}\nolimits}
\newenvironment{centre}{\begin{center}}{\end{center}}

\begin{document}
\theoremstyle{plain}
\newtheorem{thm}{Theorem}
\newtheorem{lem}[thm]{Lemma}
\newtheorem{cor}[thm]{Corollary}
\newtheorem{prop}[thm]{Proposition}

\title{Preferential attachment graphs with co-existing types of different fitnesses}
\author{Jonathan Jordan\\ University of Sheffield}

\maketitle
\begin{abstract}We extend the work of Antunovi\'{c}, Mossel and R\'{a}cz on competing types in preferential attachment models to include cases where the types have different fitnesses, which may be either multiplicative or additive.  We show that, depending on the values of the parameters of the models, there are different possible limiting behaviours depending on the zeros of a certain function.  In particular we show the existence of choices of the parameters where one type is favoured both by having higher fitness and by the type attachment mechanism, but the other type has a positive probability of dominating the network in the limit.\end{abstract}

\section{Introduction}

In \cite{AMR}, Antunovi\'{c}, Mossel and R\'{a}cz consider preferential attachment graphs with a number of competing types, with new vertices being assigned to the types in a way which depends on the types of their neighbours.  They show that, depending on the mechanism for assigning the types to the new vertices, various limiting behaviours for the types of vertices are possible, including situations where one type ends up dominating but also including situations where the types co-exist.  For many choices of type assignment mechanism, more than one limiting behaviour is possible.  In this paper, we concentrate on the case where there are two types, which is also mainly the case in \cite{AMR}.

The aim of this paper is to extend the class of attachment models studied by \cite{AMR}, and to investigate the effects of different types of attachment model on the competition between the types.  In particular we consider attachment models where the two types are treated differently, so that one type may be favoured, being seen as having higher fitness.  These models are based on the preferential attachment with multiplicative fitness model, introduced by Bianconi and Barab\'{a}si \cite{bianconi}, and the additive fitness model, introduced by Erg\"{u}n and Rodgers \cite{ergun2002}.  This can be thought of as extending the framework of \cite{AMR} to include cases where one of the types has some intrinsic advantage over the other in terms of attracting new connections.

Similarly to \cite{AMR}, we use stochastic approximation methods, both one and two dimensional, to show that there are a number of possible limits, which are stable zeros of a particular function on $[0,1]$ which depends on the choice of parameters and the nature of the fitness mechanism.  In particular, we show that, in both the multiplicative and additive fitness models, there are choices of the parameter values where one type has higher fitness and is also favoured by the type assignment mechanism, but there is positive probability of the other type dominating the network in the limit.  There are also cases with a symmetric type assignment mechanism where a less fit type is able to maintain a positive proportion of the vertices with probability $1$.  Typically there are threshold values for the ratio of the fitnesses between the types beyond which this behaviour cannot occur and the fitter type dominates almost surely.

In Section \ref{AMRreview} we review \cite{AMR} and the results of that paper.  In Section \ref{affine} we extend the results on the affine preferential attachment model, which is covered briefly in \cite{AMR}; we show that the results proved in \cite{AMR} fully extend to this model.  In Section \ref{multiplicative} we consider the model with multiplicative fitness, and in Section \ref{additive} we consider the model with additive fitness, which can be seen as a generalisation of affine preferential attachment where the additive constants are different for the two types.  In each of Sections \ref{multiplicative} and \ref{additive} we give some examples of choices of the parameters, the possible behaviours, and how phase transitions occur as the fitness values change.  Finally, in Section \ref{discuss} we discuss some further topics related to the model, including the extension to more than two colours and the link to multiple drawing P\'{o}lya urns; urns of this type have been considered by a number of recent papers including Lasmar, Mailler and Selmi \cite{lasmar} and Gao and Mahmoud \cite{gao2018} and have a natural connection to preferential attachment.

\section{The Antunovi\'{c}, Mossel and R\'{a}cz model and their results}\label{AMRreview}

The model in \cite{AMR} is a standard Barab\'{a}si-Albert model as introduced in \cite{scalefree1999}, with a new vertex connecting to $m$ existing vertices, chosen with probability proportional to their degree, with each of the $m$ vertices for a given new vertex chosen independently as in the variant of preferential attachment in \cite{scalefree7} or the ``independent model'' of \cite{bergerpa}.

More precisely, we start with a graph $G_0$, where each vertex is of one of two colours, described as red and blue; for a given vertex $v$ we set $T_v=1$ if $v$ is red and $T_v=2$ if $v$ is blue.  Throughout this paper we assume that $G_0$ contains at least one vertex of each type.  At each time step we form a new graph $G_{n+1}$ from $G_n$ by adding a single vertex and edges connecting the new vertex to $m$ vertices $W^{(n+1)}_1, \ldots, W^{(n+1)}_m$, with the $W^{(n+1)}_i$ independent of each other, conditional on $G_n$, and distributed so that for each vertex $v$ of $G_n$ $$\pr(W^{(n+1)}_i=v|\salj_n)=\frac{\deg_{G_n}(v)}{\sum_{u\in V(G_n)}\deg_{G_n}(u)}.$$  Here $\salj_n$ is the $\sigma$-algebra generated by the graphs $G_0,\ldots,G_n$ and the types of their vertices.

The model includes parameters $p_k\in[0,1]$ for each $k\in 0,1,\ldots,m$, and the new vertex $v_{n+1}$ chooses its colour by becoming red with probability $$\pr(T_{v_{n+1}}=1|\salj_n,\{W^{(n+1)}_i:1\leq i\leq m\})=p_{K_{n+1}},$$ where $K_{n+1}$ is the number of the vertices $W^{(n+1)}_1, \ldots, W^{(n+1)}_m$ which are red; otherwise the new vertex will be blue.  Hence the variation in behaviour is obtained by different choices of the $m$ and the $p_k$.  An obvious example is to let $p_k=k/m$, which in \cite{AMR} is called the \emph{linear model} and which corresponds to the new vertex picking its colour by adopting the colour of one of its neighbours chosen at random.

In \cite{AMR} it is shown that, in the linear model, the proportion of red vertices converges, and that the limiting distribution has full support and no atoms.  The proof uses the fact that, if $X_n$ and $Y_n$ are the total degrees of all red and blue vertices respectively, and $x_n=X_n/(X_n+Y_n)$, then $(x_n)_{n\in\N}$ is a martingale.

In the non-linear models the authors of \cite{AMR} define the polynomial \begin{equation}\label{poly}P(z)=\frac12\sum_{k=0}^m \binom{m}{k}z^k(1-z)^{m-k}\left(p_k-\frac{k}{m}\right),\end{equation} and they use a stochastic approximation approach to show that the proportion of red vertices converges almost surely to a stable zero or touchpoint of $P(z)$, and that any such stable zero in $[0,1]$ has positive probability of being the limit, as does any touchpoint in $(0,1)$.  (Note that in the linear case $P(z)=0$ for all $z$.)  Here a stable zero of $P(z)$ is a value $p$ such that $P(p)=0$ and there exists $\epsilon>0$ such that $P(z)>0$ for $z\in(p-\epsilon,p)\cap [0,1]$ and $P(z)<0$ for $z\in(p,p+\epsilon)\cap [0,1]$, while a touchpoint is a value $p$ such that $P(p)=0$ and there exists $\epsilon>0$ such that either $P(z)>0$ for all $z\in((p-\epsilon,p+\epsilon)\setminus \{p\})$ or $P(z)<0$ for all $z\in((p-\epsilon,p+\epsilon)\setminus \{p\})$.  Various examples are shown, including one where one of the two colours will have a proportion tending to $1$ as $n\to\infty$, and one where there are a number of possible limits, each of which involves co-existence of the two types but with different limiting proportions.

\section{Affine preferential attachment}\label{affine}

In this section we consider the same type selection process as in \cite{AMR}, but on the affine preferential attachment model, introduced by Dorogovtsev, Mendes and Samukhin \cite{dorog} and later studied rigorously in various papers including \cite{buckley, scalefree7}.  In this model, instead of existing vertices being chosen with probability proportional to their degrees we now have $$\pr(W^{(n+1)}_i=v|\salj_n)=\frac{\deg_{G_n}(v)+\alpha}{\sum_{u\in V(G_n)}(\deg_{G_n}(u)+\alpha)},$$ for some $\alpha>-m$, with the model otherwise being identical to that in Section \ref{AMRreview}.  This model is considered briefly in Section 4 of \cite{AMR}, where, using a two-dimensional stochastic approximation process, it is stated that the same results apply if $\alpha>0$ but it is suggested that there may be other possibilities if $\alpha<0$.  Below, using a slightly different method which only requires a one-dimensional stochastic approximation process, we show that the same results as in \cite{AMR} apply for all $\alpha$.

As in \cite{AMR}, let $A_n$ (resp. $B_n$) be the number of red (resp. blue) vertices in $G_n$, and let $X_n$ (resp. $Y_n$) be their total degree.  We now define $$q_n=\frac{X_n+\alpha A_n}{X_n+\alpha A_n+Y_n+\alpha B_n},$$ which is the probability that a particular edge from the new vertex connects to a red vertex, and note that $X_n+\alpha A_n+Y_n+\alpha B_n=(2m+\alpha)n+c$, where $c$ is a constant depending on the initial graph.  Below, we work with $q_n$, which defines a Markov process; note that almost sure convergence of $q_n$ to a limit as $n\to \infty$ also implies almost sure convergence of the conditional probability that a new vertex is red, and hence of the proportion of red vertices, $\frac{A_n}{A_n+B_n}$, and that as long as $\lim_{n\to\infty} q_n$ is a zero of $P$, these limits will be the same as for $q_n$.

\subsection{The linear case}

\begin{thm}
In the linear model, with $p_k=k/m$, we have that $q_n$ converges almost surely as $n\to\infty$ to a limit $q$, which is a random variable with a distribution with full support on $[0,1]$ and with no atoms.
\end{thm}

\begin{proof}
In the linear model, the probability that the new vertex in $G_{n+1}$ is red is $q_n$, and each of its $m$ neighbours is red with probability $q_n$, so $\E(A_{n+1}|\salj_n)=A_n+q_n$ and $\E(X_{n+1}|\salj_n)=X_n+2mq_n$.  Hence \begin{eqnarray*}\E(q_{n+1}|\salj_n) &=& \frac{1}{(2m+\alpha)(n+1)+c}(X_n+2mq_n+\alpha(A_n+q_n)) \\ &=& \frac{1}{(2m+\alpha)(n+1)+c}(((2m+\alpha)n+c)q_n+(2m+\alpha)q_n) \\ &=& q_n,\end{eqnarray*} so that $(q_n)$ is a $[0,1]$-valued martingale, and hence it converges a.s. to some limit $q$.

To show that the distribution of $q$ has full support on $[0,1]$ and that it has no atoms, we follow the proof of Theorem 1.1 in Section 2.2 of \cite{AMR}.  Firstly, we can bound $$(q_{n+1}-q_n)^2=\left(\frac{X_{n+1}+\alpha A_{n+1} - X_n-\alpha A_n}{(2m+\alpha)(n+1)+c}\right)\leq \frac{1}{(n+1)^2},$$ from where the proof of full support in $(0,1)$ follows exactly as in \cite{AMR}.  To show that there are no atoms in $(0,1)$, if $q_n$ is in a neighbourhood of a point $r\in (0,1)$ we can bound $\Var(q_{n+1}|\salj_n)>\frac{b}{(n+1)^2}$ for some constant $b$, from which again we can use the same argument as in \cite{AMR} to show $\pr(q_n\to r)=0$.  To show that there are no atoms at 0 or 1, the proof is again the same as in \cite{AMR} except that the comparison is to a P\'{o}lya urn with $2m+\alpha$ balls added at each step.
\end{proof}

\subsection{The non-linear case}

The following result shows that the limiting behaviour of $q_n$ in the affine non-linear model satisfies the same results as those found in \cite{AMR} for standard preferential attachment.

\begin{thm}
In the non-linear model, $q_n$ almost surely converges to a limit, which is a stable zero or touchpoint of the polynomial $P$ defined in \eqref{poly}, and all stable zeros of $P$ in $[0,1]$ and all touchpoints in $(0,1)$ have positive probability of being the limit.
\end{thm}

\begin{proof}
The probability that the new vertex is red is $$\sum_{k=0}^m p_k\binom{m}{k}q_n^k(1-q_n)^{m-k}=2P(q_n)+q_n,$$ and the probability that each neighbour of the new vertex is red is $q_n$.

So $$\E(A_{n+1}|\salj_n)=A_n+\sum_{k=0}^m p_k\binom{m}{k}q_n^k(1-q_n)^{m-k}$$ and $$\E(X_{n+1}|\salj_n)=X_n+mq_n+m\sum_{k=0}^m p_k\binom{m}{k}q_n^k(1-q_n)^{m-k}.$$  Combining these, $$\E(q_{n+1}|\salj_n)=\frac{1}{(2m+\alpha)(n+1)+c}\left[((2m+\alpha)n+c)q_n+(m+\alpha) \sum_{k=0}^m p_k\binom{m}{k}q_n^k(1-q_n)^{m-k}+mq_n\right],$$ from which we obtain \begin{eqnarray*}\E(q_{n+1}|\salj_n)-q_n &=& \frac{m+\alpha}{(2m+\alpha)(n+1)+c}\left[\sum_{k=0}^m p_k\binom{m}{k}q_n^k(1-q_n)^{m-k}-q_n\right] \\ &=& \frac{m+\alpha}{(2m+\alpha)(n+1)+c}\left[\sum_{k=0}^m\left(p_k-\frac{k}{m}\right)\binom{m}{k}q_n^k(1-q_n)^{m-k}\right] \\ &=& 2\frac{m+\alpha}{(2m+\alpha)(n+1)+c}P(q_n).\end{eqnarray*}

Hence $(q_n)_{n\in \N}$ satisfies the stochastic approximation $$q_{n+1}=q_n+\frac{2(m+\alpha)}{(2m+\alpha)(n+1)+c)}(P(q_n)+\xi_{n+1}),$$ where $\xi_{n+1}$ is defined as $\frac{(2m+\alpha)(n+1)+c)}{2(m+\alpha)}(q_{n+1}-\E(q_{n+1}|\salj_n))$ and therefore satisfies $\E(\xi_{n+1}|\salj_n)=0$.  As $|q_{n+1}-\E(q_{n+1}|\salj_n)|\leq \frac{2m+\alpha}{(2m+\alpha)(n+1)+c)}$, we have that $|\xi_{n+1}|\leq 2m+\alpha$, which is enough to tell us, using Corollary 2.7 of Pemantle \cite{pemantlesurvey}, that $q_n$ converges to the zero set of $P$, and, using Theorem 2.8 of \cite{pemantlesurvey}, that all stable zeros of $P$ have positive probability of being the limit.  That this also applies to touchpoints follows from Theorem 2.5 in \cite{AMR}.

If an unstable zero $r\in (0,1)$ then Lemma 2.7 of \cite{AMR} applies with $X_n$ replaced by $X_n+\alpha A_n$ and with $k_1$ and $k_2$ not necessarily integers, and hence the argument that $r$ is a limit with probability zero is an application of Theorem 2.9 of \cite{pemantlesurvey} which is essentially identical to the proof of Theorem 1.4 of \cite{AMR} in this context.  For unstable zeros of $P$ at 0 or 1, again the proof of Theorem 1.4 in \cite{AMR} works, except that the  P\'{o}lya urn used for comparison adds $2m+\alpha$ balls at each step.\end{proof}

\section{Multiplicative fitness}\label{multiplicative}

We now move to considering extensions of the model of \cite{AMR} where the types interact differently with the preferential attachment mechanism.  We first consider a multiplicative fitness model, which is inspired by that introduced by Bianconi and Barab\'{a}si \cite{bianconi} and studied in more detail by Borgs et al \cite{borgsfitness}, and which has different fitnesses for the two types.

Specifically, we assume that red vertices have fitness $1$ and blue vertices have fitness $\phi$ for some $\phi$ which we assume to be greater than $1$, so that red vertices are chosen with probability proportional to their degree and blue vertices are chosen with probability proportional to their degree times $\phi$.  We also allow a constant $\alpha>-m$ to be added to the degrees, as in Section \ref{affine}.  Formally, the model is the same as that described in Section \ref{AMRreview} except that we now have $$\pr(W^{(n+1)}_i=v|\salj_n)=\frac{(\deg_{G_n}(v)+\alpha)\phi^{T_v-1}}{\sum_{u\in V(G_n)}(\deg_{G_n}(u)+\alpha)\phi^{T_u-1}}.$$  The model differs from the models of \cite{bianconi, borgsfitness} in that the assignment of types to vertices is now based on the types of their neighbours rather than independent as in those papers.

Let $$x_n=\frac{X_n+\alpha A_n}{X_n + Y_n + \alpha(A_n+B_n)}=\frac{X_n+\alpha A_n}{(2m+\alpha)n+c},$$ where we define $c$ so that $X_0+Y_0+\alpha(A_0+B_0)=c$.     Define \begin{equation}\label{pphi}P^M(x)=\frac{2(m+\alpha)}{2m+\alpha}P\left(\frac{x}{x+\phi(1-x)}\right)+\left(\frac{x}{x+\phi(1-x)}-x\right),\end{equation} with $P$ as defined previously, and note that $P^M$ is a rational function with numerator having degree at most $m+1$, and that it cannot be identically zero for any choice of the $p_k$ and $\alpha$ if $\phi>1$.  In the special case of the linear model we have $$P^M(x)=\frac{(1-\phi)x(1-x)}{x+\phi(1-x)}.$$

\begin{lem}The sequence $(x_n)_{n\in \N}$ follows a one-dimensional stochastic approximation process associated to a flow given by $P^M$.\end{lem}
\begin{proof}
Conditional on $\salj_n$, the probability that a single vertex chosen is red is $\frac{x_n}{x_n+\phi  y_n}$.  Hence we can write \begin{eqnarray*}\E(x_{n+1}|\salj_n)-x_n &=& \frac{m\frac{x_n}{x_n+\phi  y_n}+(m+\alpha)\sum_{k=0}^m\binom{m}{k}p_k\left(\frac{x_n}{x_n+\phi  y_n}\right)^k\left(\frac{\phi y_n}{x_n+\phi  y_n}\right)^{m-k}-(2m+\alpha)x_n}{(2m+\alpha)(n+1)+c}\\ &=& \frac{(2m+\alpha)\frac{x_n}{x_n+\phi  y_n}+(m+\alpha)\left(2P\left(\frac{x_n}{x_n+\phi  y_n}\right)\right)- (2m+\alpha)x_n}{(2m+\alpha)(n+1)+c}\\ &=&
\frac{1}{(n+1)+c/(2m+\alpha)}P^M(x_n),\end{eqnarray*} and so $(x_n)_{n\in \N}$ satisfies the stochastic approximation $$x_{n+1}=x_n+\frac{1}{(n+1)+c/(2m+\alpha)}(P^M(x_n)+\xi_{n+1}),$$ where $\xi_{n+1}$ is defined as $((n+1)+c/(2m+\alpha))(x_{n+1}-\E(x_{n+1}|\salj_n))$ and therefore satisfies $\E(\xi_{n+1}|\salj_n)=0$.\end{proof}

\begin{thm}\label{multthm}As $n\to\infty$, $x_n$ converges almost surely to a zero of the function $P^M$ defined in \eqref{pphi}.  Furthermore, any stable zero of $P^M$ in $[0,1]$ and any touchpoint in $(0,1)$ has positive probability of being the limit, and any unstable zero in $[0,1]$ has probability zero of being the limit.\end{thm}

\begin{proof}That the process almost surely converges to a zero of $P^M$ follows from Corollary 2.7 of Pemantle \cite{pemantlesurvey}, that any stable zero of $P^M$ has positive probability of being the limit follows from Theorem 2.8 of \cite{pemantlesurvey}, and as before Theorem 2.5 of \cite{AMR} implies that this also applies to touchpoints.  That unstable zeros in $(0,1)$ are limits with probability zero follows from Theorem 2.9 of \cite{pemantlesurvey}, again using a modification of Lemma 2.7 of \cite{AMR} for the noise condition.

To show that unstable zeros at the endpoint $1$ are limits with probability zero, we adapt the argument for the equivalent part of the proof of Theorem 1.4 in \cite{AMR}.  This argument shows that, for some sufficiently small $\epsilon>0$, as long as the process $(x_n)$ remains in $(1-\epsilon,1]$ the values $X_n$ can be coupled to a P\'{o}lya urn process $(\bar{X}_n)$ which adds $2m+\alpha$ balls at each step in such a way that $X_n\leq_{\text{lcx}} \bar{X}_n$ in the \emph{increasing convex order}; that is to say that $\E(\psi(X_n))\leq \E(\psi(\bar{X}_n))$ for all increasing convex functions for which the expectations exist.

To adapt the argument we need to show that, conditional on $x_n=\bar{x}_n=r$, $X_{n+1}\leq_{\text{lcx}}\bar{X}_{n+1}$; once we do this we can apply the same induction argument as in \cite{AMR}.  To show that, Lemma 2.9 of \cite{AMR} shows that it is enough to show that \begin{equation}\label{lcx}\E(X_{n+1}|x_n=r)\leq \E(\bar{X}_{n+1}|\bar{x}_n=r),\end{equation} where $\bar{x}_n$ denotes the proportion of red balls in the P\'{o}lya urn, and that the distribution functions $F$ and $G$ of $X_{n+1}$ and $\bar{X}_{n+1}$ respectively, both conditioned on $x_n=\bar{x}_n=r$, satisfy the property that if $t_1<t_2$ and $G(t_1)<F(t_1)$ then $G(t_2)\leq F(t_2)$.  That \eqref{lcx} holds follows from the fact that, as we are assuming $1$ is an unstable zero of $P^M$, $P^M(x)<0$ for $x\in(1-\epsilon, 1)$ for some $\epsilon$, and the property involving the distribution functions then follows, as in \cite{AMR}, from the increments of the urn process, $\bar{X}_{n+1}-\bar{X}_n$, being concentrated on $\{0,2m+\alpha\}$ while $X_{n+1}-X_n$ is supported on the interval $[0,2m+\alpha]$.

An analogous argument shows that unstable zeros at the endpoint $1$ are limits with probability zero.
\end{proof}

This allows us to conclude that in the linear model the type with the higher fitness dominates.
\begin{cor}\label{mult_linear}In the linear model with $\phi>1$ the proportion of blue vertices converges to $1$ almost surely.
\end{cor}

\begin{proof}Because $x+\phi(1-x)>1$ for $x<1$, we have that $P^M(x)<0$ for $0<x<1$, and so $0$ is the only stable zero of $P^M$.  Hence, by Theorem \ref{multthm}, $x_n\to 0$ almost surely, giving the result.\end{proof}

\subsection{Discussion and examples}\label{multdiscuss}

In these examples we concentrate on cases where $p_0=0$ and $p_m=1$; if $p_0>0$ then dominance by blue vertices is impossible, and similarly if $p_m<1$ dominance by red vertices is impossible.  We also assume $\phi>1$, meaning that blue vertices have higher fitness.  We first consider the case when $\alpha=0$, and then consider the effect of varying $\alpha$.

\subsubsection{The case $\alpha=0$}

If $p_0=0$ and $p_m=1$ then both $0$ and $1$ are zeros of $P^M$.  It is not hard to show that $0$ is a stable zero, indicating positive probability of blue dominance, if $\phi>\frac12 m\left(p_1-\frac1m\right)+1=\frac12(mp_1+1),$ and an unstable zero if $\phi<\frac12(mp_1+1)$; if $\phi=\frac12(mp_1+1)$ the stability depends on the other $p_k$.  Similarly $1$ is a stable zero, indicating positive probability of red dominance, if $\phi<\frac{2}{m(1-p_{m-1})+1},$ and an unstable zero if $\phi>\frac{2}{m(1-p_{m-1})+1}$; again if we have equality the stability depends on the other $p_k$.  Note that if $\phi>2$ then $\phi>\frac{2}{m(1-p_{m-1})+1}$ always holds, so red dominance cannot have positive probability in that case.

When $m=2$ the fixed point at $0$ is stable if $p_1<\phi-\frac12$, which is always true if $\phi>\frac32$, and the fixed point at $1$ is stable if $p_1>\frac32-\frac{1}{\phi}$, which requires $\phi<2$.  Hence if $\phi>2$ blue vertices dominate almost surely, whatever the value of $p_1$, while if $\frac32<\phi<2$, then there is always a positive probability of blue domination, while red domination happens with positive probability if $p_1>\frac32-\frac{1}{\phi}$.
If $1<\phi<\frac32$, then red domination happens with positive probability if $p_1>\frac32-\frac{1}{\phi}$ and blue domination happens with positive probability if $p_1<\phi-\frac12$.  One of these criteria is always satisfied, as $\frac32-\frac{1}{\phi}<\phi-\frac12$ in this range.  As the numerator of $P^M$ is a cubic with roots at $0$ and $1$, at least one of which is stable, there can be no other stable fixed point in $(0,1)$, so with probability $1$ one of the types dominates in the limit.

When $m=3$, we consider some examples of how the behaviour varies with $\phi$ for different choices of $p_1$ and $p_2$.

\begin{enumerate}
\item $m=3$, $p_0=0,p_1=p_2=\half,p_3=1$.  In this case there is almost sure co-existence in the case where the two types have the same fitness, $\phi=1$.  If $\phi<5/4$, the zeros of $P^M$ at $0$ and $1$ are both unstable, with a stable zero in $(0,1)$, so this remains true, but with the limiting proportion of red decreasing with $\phi$.  For $\phi\geq 5/4$, the zero of $P^M$ at $0$ becomes stable, and there is no stable zero in $(0,1)$, so blue dominance occurs almost surely.  See the left plot in Figure \ref{fig} for plots of $P^M$ either side of the phase transition at $\phi=5/4$.
\item $m=3$, $p_0=0,p_1=\frac14,p_2=\frac34,p_3=1$.  In this case, if the types have the same fitness then almost surely one of the two types dominates, with both having positive probability of doing so.  If $\phi<8/7$, then both $0$ and $1$ are stable zeros of $P^M$ and other zeros are unstable, so this remains the case.  If $\phi\geq 8/7$ then the zero at $1$ becomes unstable, and the other zeros are outside $(0,1)$, so blue dominance occurs with probability $1$.
\item $m=3$, $p_0=0,p_1=p_2=p_3=1$.  In this case the type assignment mechanism has a bias towards red, as the new vertex will be red if it connects to any red vertices.  There are two phase transitions in $\phi$.  For $\phi\leq 2$ the zero of $P^M$ at $1$ remains stable and that at $0$ remains unstable, indicating that the bias in type assignment mechanism dominates the fitness effect so that red dominance still occurs almost surely.  For $2<\phi\leq \frac{3+\sqrt{2}}{2}$ the zero at $1$ becomes unstable and that at $0$ becomes stable, but there is also a stable zero in $(0,1)$, so that both blue dominance and co-existence have positive probability.  For $\phi>\frac{3+\sqrt{2}}{2}$, the only stable zero of $P^M$ is at $0$, so the fitness effect now dominates and blue dominance occurs almost surely.  See the middle plot in Figure \ref{fig} for plots of $P^M$ for values of $\phi$ above, below and between the two phase transitions.
\item \label{biasblue} $m=3$, $p_0=p_1=0,p_2=9/10,p_3=1$.  Here there is a slight bias in the type assignment mechanism towards blue.  However, for $\phi<20/13$, both the zeros of $P^M$ at $0$ and $1$ remain stable.  Hence, in this case we see that there is positive probability of the less fit type dominating, even though there is a bias in the type assignment mechanism against it as well as the fitness effect.  For $\phi\geq 20/13$ the only stable zero of $P^M$ in $[0,1]$ is at $0$, so blue dominance occurs almost surely.  See the right plot in Figure \ref{fig} for plots of $P^M$ either side of the phase transition at $\phi=20/13$. We investigate this example further by simulation in Section \ref{limitdist}.
\end{enumerate}

\begin{figure}
\begin{centre}
\begin{tabular}{ccc}\scalebox{0.3}{{\includegraphics{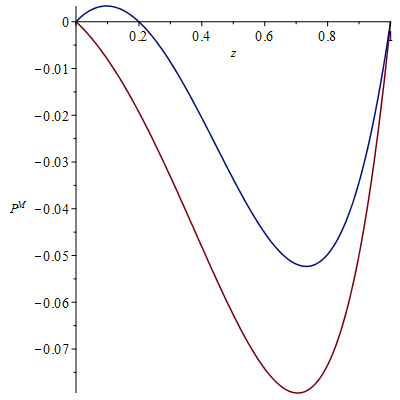}}} & \scalebox{0.3}{{\includegraphics{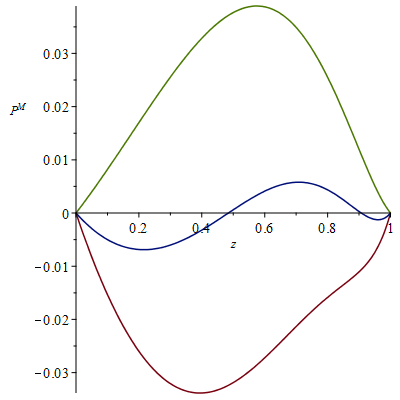}}} & \scalebox{0.3}{{\includegraphics{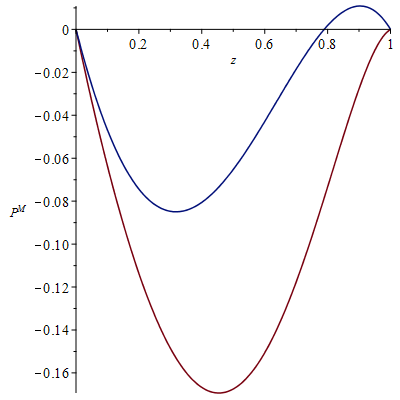}}} \end{tabular}
\caption{Plots of $P^M$ for $m=3$.  Left plot: $p_0=0,p_1=p_2=\half,p_3=1$, upper curve $\phi=7/6$, lower curve $\phi=4/3$.  Middle plot: $p_0=0,p_1=p_2=p_3=1$, top curve $\phi=13/7$, middle curve $\phi=15/7$, bottom curve $\phi=17/7$.  Right plot: $p_0=p_1=0,p_2=9/10,p_3=1$, upper curve $\phi=7/6$, lower curve $\phi=5/3$.  Plots created using Maple. \label{fig}}
\end{centre}
\end{figure}

\subsubsection{The effect of varying $\alpha$}

In Section \ref{affine} we showed that varying $\alpha$ does not change the results of \cite{AMR} in the model without fitness.  By considering the form of $P^M$, we can see that this is not the case in the multiplicative fitness model; indeed as $\alpha\to -m$ with $\phi$ and the $p_k$ fixed we get $P^M(x)\to \frac{x}{x+\phi(1-x)}-x=\frac{(1-\phi)x(1-x)}{x+\phi(1-x)}$, which is negative for all $x\in (0,1)$ for $\phi>1$.  Hence, for any choice of $\phi>1$ and the $p_k$ with $p_0=0$ and $p_m=1$, we get that if $\alpha$ is small enough $P^M$ is negative on $(0,1)$ and so blue dominance occurs almost surely.

\section{Additive fitness}\label{additive}

We now extend the model of Section \ref{affine} by allowing the two types to have different values of $\alpha$, which can be seen as corresponding to different fitnesses of the two types: if $\alpha_2>\alpha_1$, blue vertices are ``fitter'' as they are more likely to be chosen than red vertices of the same degree.  This model thus resembles the additive fitness model of Erg\"{u}n and Rodgers \cite{ergun2002}, whose degree distribution is analysed in detail by Bhamidi \cite{bhamidi2007}, except that, as in the multiplicative model of Section \ref{multiplicative}, the fitnesses of the new vertices correspond to their types, and hence are influenced by the types of their neighbours, rather than being independent random variables as in the model of \cite{ergun2002}.

The model is as described in Section \ref{AMRreview}, except that now $$\pr(W^{(n+1)}_i=v|\salj_n)=\frac{\deg_{G_n}(v)+\alpha_{T_v}}{\sum_{u\in V(G_n)}(\deg_{G_n}(u)+\alpha_{T_u})},$$  In this section, we assume $\alpha_1\neq \alpha_2$ and, to avoid degeneracies, we assume $\alpha_1,\alpha_2>-m$.  Without loss of generality we assume $\alpha_2>\alpha_1$, so that blue vertices are fitter than red.

The probability of a new edge at time $n+1$ connecting to an existing red vertex is now $$q_n=\frac{X_n+\alpha_1 A_n}{X_n+\alpha_1 A_n+Y_n+\alpha_2 B_n}.$$  Define $x_n=\frac{X_n+\alpha_1 A_n}{n}$ and $y_n=\frac{Y_n+\alpha_2 B_n}{n}$, so that $q_n=\frac{x_n}{x_n+y_n}$.  We consider the bivariate process $((x_n,y_n))_{n\in\N}$.

Define $$P^A(z)=(\alpha_1-\alpha_2)z(1-z)+[2(m+\alpha_1)+2(\alpha_2-\alpha_1)z]P(z).$$  Note that $P^A$ is a polynomial, and that it is not identically zero unless $\alpha_1=\alpha_2$.  Our aim is to prove the following theorem.

\begin{thm}\label{addthm}
	Assume that we have $0<p_k<1$ for $0<k<m$.  Then, almost surely, $(x_n,y_n)\to (x,y)$ as $n\to\infty$, where $(x,y)$ is a (possibly random limit) such that $P^A\left(\frac{x}{x+y}\right)=0$.  Furthermore, any zero $q$ of $P^A$ which is stable in the sense that $(P^A)'(q)<0$ has positive probability of having $\frac{x}{x+y}=q$.
\end{thm}

To move towards proving Theorem \ref{addthm}, we first consider the state space of $(x_n,y_n)$; we first note that, for all $n\geq 0$, \begin{equation}\label{easybound}2m+\alpha_1 \leq x_n+y_n\leq 2m+\alpha_2.\end{equation}  We then observe that as $$Y_n+\alpha_2 B_n=2mn-X_n+\alpha_2(n-A_n),$$ we have \begin{equation}\label{ybound}y_n=2m+\alpha_2-x_n-(\alpha_2-\alpha_1)\frac{A_n}{n}.\end{equation}  Furthermore, as each vertex has degree at least $m$, we have \begin{equation}\label{Abound}A_n\leq \frac{1}{m} X_n,\end{equation} which implies $$A_n\left(1+\frac{\alpha_1}{m}\right)\leq \frac{1}{m}(X_n+\alpha_1 A_n),$$ and hence $$\frac{A_n}{n}\leq \frac{x_n}{m+\alpha_1}.$$  Combining this with \eqref{ybound} gives \begin{equation}\label{ybound2} y_n \geq 2m+\alpha_2-x_n\left(\frac{m+\alpha_2}{m+\alpha_1}\right).\end{equation}
An analogous argument gives \begin{equation}\label{xbound}x_n\leq 2m+\alpha_1-y_n\left(\frac{m+\alpha_1}{m+\alpha_2}\right).\end{equation}
The state space is then the parallelogram, which we call $D$, given by \eqref{easybound}, \eqref{ybound2} and \eqref{xbound}.

We note that \eqref{Abound} is tight if $p_i\in (0,1)$ for all $i\in\{0,1,\ldots,m\}$.  If $p_0=0$ and $p_m=1$, as we usually assume, then \eqref{Abound} can be improved to $$A_n\leq \frac{1}{m+1} X_n,$$ as now any red vertex must have at least one red neighbour, and this is tight if the other $p_i\in (0,1)$.  Hence \eqref{ybound2} and \eqref{xbound} are modified to \begin{equation}\label{ybound3} y_n \geq 2m+\alpha_2-x_n\left(\frac{m+1+\alpha_2}{m+1+\alpha_1}\right)\end{equation} and \begin{equation}\label{xbound2}x_n\leq 2m+\alpha_1-y_n\left(\frac{m+1+\alpha_1}{m+2+\alpha_2}\right).\end{equation} Define $D_0$ to be the parallelogram given by \eqref{easybound}, \eqref{ybound3} and \eqref{xbound2}; in this case the state space can be thought of as $D_0$ rather than $D$.

Now define \begin{eqnarray*}F_1(x,y) &=& m\frac{x}{x+y}+(m+\alpha_1)\left(\sum_{k=0}^m p_k\binom{m}{k}\left(\frac{x}{x+y}\right)^k\left(\frac{y}{x+y}\right)^{m-k}\right)-x \\ &=& (2m+\alpha_1)\frac{x}{x+y}+2(m+\alpha_1)P\left(\frac{x}{x+y}\right)-x\end{eqnarray*} and \begin{eqnarray*}F_2(x,y)&=& m\frac{y}{x+y}+(m+\alpha_2)\left(\sum_{k=0}^m  (1-p_k)\binom{m}{k}\left(\frac{x}{x+y}\right)^k\left(\frac{y}{x+y}\right)^{m-k}\right)-y\\&=& (2m+\alpha_2)\frac{y}{x+y}-2(m+\alpha_2)P\left(\frac{x}{x+y}\right)-y,\end{eqnarray*} and furthermore define $F:D\to\R^2$ by $F(x,y)=(F_1(x,y),F_2(x,y))$.  \begin{lem}We have that $((x_n,y_n))_{n\in\N}$ follows a bivariate stochastic approximation process associated to the flow defined by $F$.\end{lem}
\begin{proof}
We have $$\E(X_{n+1}+\alpha_1 A_{n+1}|\salj_n)=X_n+\alpha_1 A_n+mq_n+(m+\alpha_1)\sum_{k=0}^m p_k\binom{m}{k}q_n^k(1-q_n)^{m-k},$$ giving \begin{eqnarray*}\E(x_{n+1}|\salj_n) &=& x_n+\frac{1}{n+1}\left(mq_n+(m+\alpha_1)\sum_{k=0}^m p_k\binom{m}{k}q_n^k(1-q_n)^{m-k}-x_n\right)\\ &=& x_n+\frac{1}{n+1}\left((2m+\alpha_1)q_n+2(m+\alpha_1)P(q_n)-x_n \right) \\&=& x_n+\frac{1}{n+1}F_1(x_n,y_n),\end{eqnarray*} and similarly we get $$\E(y_{n+1}|\salj_n)=y_n+\frac{1}{n+1}F_2(x_n,y_n).$$  Thus $$(x_{n+1},y_{n+1})=(x_n,y_n)+\frac{1}{n+1}(F(x_n,y_n)+(\xi_{n+1},\eta_{n+1})),$$ where $\xi_{n+1} = (n+1)(x_{n+1}-\E(x_{n+1}|\salj_n))$ and $\eta_{n+1} = (n+1)(y_{n+1}-\E(y_{n+1}|\salj_n))$, so that $\E((\xi_{n+1},\eta_{n+1})|\salj_n)=(0,0)$.\end{proof}

We now analyse this stochastic approximation to prove Theorem \ref{addthm}.

\textit{Proof of Theorem \ref{addthm}}.  We construct a Lyapunov function for the flow defined by $F$ as follows.  Let $((x(t),y(t)))_{t\in \R^{+}}$ be a trajectory of the flow, and define $q(t)=\frac{x(t)}{x(t)+y(t)}$.

Note that $$\frac{d}{dt}q(t)=\frac{y(t)F_1(x(t),y(t))-x(t)F_2(x(t),y(t))}{(x(t)+y(t))^2}=\frac{P^A(q(t))}{x(t)+y(t)},$$ and define $L_1(z)=-\int_1^z P^A(u)\; du$.  Then $$\frac{d}{dt}L_1(q(t))=-(x(t)+y(t))^{-1}\left(P^A(q(t))\right)^2\leq 0,$$ with equality only when $P^A(q(t))=0$.  Let $S_1=\inf_{(x,y)\in D}(x+y)^{-1}$, so that $\frac{d}{dt}L_1(q(t))\leq -S_1 \left(P^A(q(t))\right)^2$.

Define $$\ell(x,y)=(2m+\alpha_1)(2m+\alpha_2)+2m(\alpha_1-\alpha_2)P\left(\frac{x}{x+y}\right)-(2m+\alpha_2)x-(2m+\alpha_1)y,$$ so that $$\frac{d}{dt}((2m+\alpha_2)x(t)+(2m+\alpha_1)y(t))=\ell(x(t),y(t)),$$ and define $L_2(x,y)=(\ell(x,y))^2$.

Then $\frac{d}{dt}L_2(x(t),y(t))$ is \begin{eqnarray*} & & 2\ell(x(t),y(t))\left(F_1(x(t),y(t))\left(2m(\alpha_1-\alpha_2)\frac{y(t)}{(x(t)+y(t))^2}P'(q(t))-(2m+\alpha_2)\right) \right)\\ & & + 2\ell(x(t),y(t))\left(
\left(F_2(x(t),y(t))\left(-2m(\alpha_1-\alpha_2)\frac{x(t)}{(x(t)+y(t))^2}P'(q(t))-(2m+\alpha_1)\right) \right) \right)\\ &=&  2\ell(x(t),y(t))\left(-(2m+\alpha_1)(2m+\alpha_2)+2m(\alpha_2-\alpha_1)P(q(t))+(2m+\alpha_2)x+(2m+\alpha_1)y\right) \\ & & +\; 4m\ell(x(t),y(t))(\alpha_1-\alpha_2)P'(q(t))
\left(\frac{x(t)y(t)}{(x(t)+y(t))^3}(\alpha_1-\alpha_2)+P(q(t))\frac{2(m+\alpha_1)y(t)+2(m+\alpha_2)x(t)}{(x(t)+y(t))^2}\right) \\ &=& -2\left((\ell(x(t),y(t)))^2-\frac{2m(\alpha_1-\alpha_2)}{x(t)+y(t)}P'(q(t))\ell(x(t),y(t))P^A(q(t))\right).\end{eqnarray*}  As we assume $\alpha_1<\alpha_2$, define $$S_2=\sup_{(x,y)\in D}m(\alpha_2-\alpha_1)P'\left(\frac{x}{x+y}\right)\frac{1}{x+y};$$ then $$\frac{d}{dt}L_2(x(t),y(t))\leq -2\left((\ell(x(t),y(t)))^2-2S_2\ell(x(t),y(t))P^A(q(t))\right).$$  Define $$L(x,y)=L_2(x,y)+2\frac{S_2^2}{S_1}L_1\left(\frac{x}{x+y}\right),$$ so that we have $$\frac{d}{dt}L(x(t),y(t))\leq -2\left(\ell(x(t),y(t))+S_2P^A(q(t))\right)^2\leq 0,$$ with equality in the right inequality only when $\ell(x(t),y(t))=0$ and $P^A(q(t))=0$.  Hence $L$ is a Lyapunov function for $F$ with stationary points $(x,y)$ given by $\ell(x,y)=0$ and $P^A\left(\frac{x}{x+y}\right)=0$; therefore our stochastic approximation process will converge to one of these by Proposition 2.18 of Pemantle \cite{pemantlesurvey}.

If $(x,y)$ is a stationary point of $L$ such that $q=\frac{x}{x+y}$ is a stable zero of $P^A$ in the sense that $P^A(q)=0$ and $(P^A)'(q)<0$, then $(x,y)$ is a local minimum of $L$.  Hence it is a limit of the process with positive probability, by Theorem 2.16 of Pemantle \cite{pemantlesurvey}.

Non-convergence to unstable points with $\frac{x}{x+y}\in(0,1)$ follows from the general non-convergence result Theorem 9.1 in Bena\"{\i}m \cite{benaim}.  Condition (iii) of that result, which requires a constant $b$ such that $\E(((\xi_{n+1},\eta_{n+1})\cdot v)^+|\salj_n)>b$ within a neighbourhood of the unstable point for any unit vector $v$, follows from our assumption that $0<p_k<1$ for $0<k<m$.

To show that $q_n$ converges to $1$ with probability zero if $P^A(1)=0$ and $(P^A)'(1)>0$, we use a similar argument to the previous coupling to a P\'{o}lya urn in Section \ref{multiplicative}.  For some sufficiently small $\epsilon$, we compare the process when $q_n\in (1-\epsilon,1]$ to an urn process which, given $\bar{X}_n$ red balls and $\bar{Y}_n$ blue balls, adds $2m+\alpha_1$ red balls with probability $\frac{\beta \bar{X}_n}{\beta \bar{X}_n+\bar{Y}_n}$ and adds $2m+\alpha_2$ blue balls with probability $\frac{\bar{Y}_n}{\beta \bar{X}_n+\bar{Y}_n}$.

The previous argument then works as long as, for some $\beta$, we can show that $$\E(q_{n+1}|\salj_n)\leq \E\left(\frac{\bar{X}_{n+1}}{\bar{X}_{n+1}+\bar{Y}_{n+1}}\large{|}\frac{\bar{X}_{n}}{\bar{X}_{n}+\bar{Y}_{n}}=r\right)$$ when $q_n=r$ with $r\in (1-\epsilon,1)$ , and that $\frac{\bar{X}_{n}}{\bar{X}_{n}+\bar{Y}_{n}}\to 1$ with probability zero.  To do this, we let $\beta=\frac{2m+\alpha_2}{2m+\alpha_1}$, and count the number of times we have added red balls and blue balls respectively as $\hat{X}_n=\bar{X}_n/(2m+\alpha_1)$ and $\hat{Y}_n=\bar{Y}_n/(2m+\alpha_2)$.  Then, as $$\frac{\beta \bar{X}_n}{\beta \bar{X}_n+\bar{Y}_n}=\frac{\hat{X}_n}{\hat{X}_n+\hat{Y}_n},$$ the process $(\hat{X}_n,\hat{Y}_n)_{n\in \N}$ follows a standard P\'{o}lya urn.  This ensures that $\pr\left(\frac{\bar{X}_{n}}{\bar{X}_{n}+\bar{Y}_{n}}\to 1\right)=0$ and that $$\E\left(\frac{\bar{X}_{n+1}}{\bar{X}_{n+1}+\bar{Y}_{n+1}}\large{|}\frac{\bar{X}_{n}}{\bar{X}_{n}+\bar{Y}_{n}}=r\right)=r,$$ which gives us what we need as $F^A(r)<0$ by assumption.

\qed

As with the multiplicative fitness model, in the linear model we get almost sure dominance by the fitter type.
\begin{cor}\label{add_linear}In the linear model with $\alpha_2>\alpha_1$ the proportion of blue vertices converges to $1$ almost surely.
\end{cor}

\begin{proof}Here $P^A(z)=(\alpha_1-\alpha_2)z(1-z)$, which is negative on $(0,1)$ if $\alpha_2>\alpha_1$, and has a stable zero at $0$ and an unstable one at $1$.  Hence the result follows from Theorem \ref{addthm}.\end{proof}

\subsection{Discussion and examples}

As in Section \ref{multiplicative}, in these examples we concentrate on cases where $p_0=0$ and $p_m=1$; if $p_0>0$ then dominance by blue vertices is impossible, and similarly if $p_m<1$ dominance by red vertices is impossible.

First, we consider the general case where $m=2$.  There are always two stationary points at $(4+\alpha_1,0)$ and $(0,4+\alpha_2)$, corresponding to dominance of the two types.  There is also a stationary point with $\frac{x}{x+y}=\frac{2\alpha_1p_1-\alpha_2+4p_1-2}{(\alpha_2-\alpha_1)(1-2p_1)}$ where this is within $(0,1)$, but this is never stable.  The stationary point $(0,4+\alpha_2)$ is stable if $\alpha_1-\alpha_2+(p_1-\half)(2\alpha_1+4)<0$, that is if $$p_1<\half+\frac{\alpha_2-\alpha_1}{2\alpha_1+4}.$$  If $\alpha_2>2(\alpha_1+1)$, this always holds for all values of $p_1$.  The stationary point $(4+\alpha_1,0)$ is stable if $\alpha_2-\alpha_1-(p_1-\half)(2\alpha_2+4)<0$, that is if $$p_1>\half+\frac{\alpha_2-\alpha_1}{2\alpha_2+4}.$$  Hence there is a range where $\half+\frac{\alpha_2-\alpha_1}{2\alpha_2+4}<p_1<\half+\frac{\alpha_2-\alpha_1}{2\alpha_1+4}$ where both stationary points are stable, and hence limits with positive probability; outside this range there is only one stable fixed point.  Co-existence is not stable for $m=2$ if $p_0=0$ and $p_2=1$, unless $\alpha_1=\alpha_2$.

When $m=3$, we give a couple of examples of how the values of $\alpha_1$ and $\alpha_2$ affect the behaviour for some specific values of the $p_k$.  As elsewhere, we assume $\alpha_2>\alpha_1$.
\begin{enumerate}
\item $m=3$, $p_0=0,p_1=p_2=\half,p_3=1$.  Here there are zeros of $P^A$ at $0$ and $1$.  Both are unstable if $\alpha_2<\frac32 (\alpha_1+1)$, and in this case there is a single stable zero in $(0,1)$, indicating almost sure co-existence.  If $\alpha_2>\frac32 (\alpha_1+1)$ then the zero at $0$ is stable, and this is the only stable zero in $[0,1]$, indicating almost sure blue dominance.  So as in the multiplicative fitness case we see a condition on the relationship between the fitnesses determining whether we get co-existence or blue dominance.
\item $m=3$, $p_0=p_1=0,p_2=9/10,p_3=1$.  Again, there are zeros of $P^A$ at $0$ and $1$.  Both are stable if $\alpha_1>\frac{3\alpha_2-21}{10}$, while if $\alpha_1<\frac{3\alpha_2-21}{10}$ the zero at $0$ is the only stable zero in $[0,1]$.  So in the latter case blue dominance happens almost surely, whereas in the latter case (for example if $\alpha_1=0$ and $\alpha_2=1$) we see that both types have a positive probability of dominance although the blue type is fitter and is also slightly favoured by the type assignment mechanism.

\end{enumerate}

\section{Further discussion}\label{discuss}

\subsection{More than two types}

A natural extension to the models considered in this paper is to consider them with more than two types of vertex.  Antunovi\'{c}, Mossel and R\'{a}cz discuss this extension for their model in Section 3 of \cite{AMR}, where they show that in the linear model their results can be easily extended to any number of types but that in non-linear models the stochastic approximation involves a multi-dimensional function about which it is hard to state general results.  This is reinforced by Haslegrave and Jordan in \cite{threetypes}, where an example of the original model of \cite{AMR} with three types is discussed and shown to have cycling behaviour in the limit, with no convergence to a fixed point.

In our setting with the types having different fitnesses, the linear model can be handled in a similar way to that in \cite{AMR}.  If all types but the fittest have the same fitness, then they can be combined into a single type, and Corollary \ref{mult_linear} (for multiplicative fitness) or Corollary \ref{add_linear} (for additive fitness) can then be applied to show that we get almost sure dominance by the fittest type.  In the more general case where the fitnesses may all be different, we can similarly couple the linear model to a two-type model where all types except the fittest are given the fitness of the second largest type; this coupling can only decrease the proportion of the total degree associated to the fittest type, so we can apply the same results.

For non-linear models, having more than two types would involve the generalisation of the functions $P^M$ and $P^A$ to multi-dimensional ones, which can have a variety of forms according to the choice of the parameters.  As a result, the situation is similar to the one in \cite{AMR} in that it is hard to state general results, and we might expect examples along the lines of that in \cite{threetypes}, where the associated differential equations do not converge to a fixed point.

\subsection{P\'{o}lya urns with fitness and multiple draws}

There is a natural connection between preferential attachment schemes and urn processes.  In the case of our model and the model of \cite{AMR}, the connection is to generalisations of P\'{o}lya urn schemes with multiple draws, about which there are a number of recent papers.  Kuba and Mahmoud \cite{kuba} consider a two-colour urn where, at each step, $m$ balls are drawn from the urn and the numbers of balls of each colour added to the urn depend (in a deterministic way) on the numbers of balls of each colour drawn.  They concentrate on a special case which produces a martingale, reminiscent of the linear model in \cite{AMR}, and prove a central limit theorem in this setting.  Lasmar, Mailler and Selmi \cite{lasmar} use stochastic approximation to extend the results to more general cases; as in \cite{AMR} and this paper the limiting behaviour is associated to the limiting behaviour of a certain differential equation.  Gao and Mahmoud \cite{gao2018} extend the model to allow for a random replacement matrix, and also use stochastic approximation to prove convergence results.

The clearest link to these papers is with the multiplicative model when $\alpha=0$.  In this case the total degrees of the two colours, $(X_n,Y_n)$, can be seen as following an urn scheme of this type where the two colours have different fitnesses or, in urn terminology, activities.  This means that the blue balls have weight $\phi$ and the red balls weight $1$, and balls are drawn with probability proportional to their weights.  Drawing $m$ balls then corresponds to drawing $m$ random edge ends with the same weights, which corresponds to our multiplicative fitness model.  The urn model corresponding to our model then has a random replacement matrix with a particular distribution, where all the balls added are of the same colour, corresponding to the colour of the new vertex in the graph model.  Our results in Theorem \ref{multthm} can thus be seen as an extension of the results of \cite{gao2018,kuba,lasmar} to the setting with fitnesses and this specific choice of replacement matrix.

\subsection{Distribution of the limit}\label{limitdist}

In cases where the function $P^M$ or $P^A$ has more than one stable zero, the results in Theorems \ref{multthm} and \ref{addthm} tell us that each stable zero is a limit with positive probability but do not tell us the actual probabilities that the different zeros are limits.  Finding the actual distribution of the limit in cases like this is a hard problem in general, and we expect that the distribution will usually depend on the initial graph.

We investigate this problem for Example \ref{biasblue} from Section \ref{multdiscuss} by simulation.  Here $m=3$, $p_0=p_1=0$, $p_2=9/10$ and $p_3=1$, giving a small bias towards blue in the type assignment mechanism, and we saw that for $\phi<20/13$ both convergence to 0 (blue domination) and convergence to 1 (red domination) have positive probability.  For each of a range of values of $\phi$ in $[1,20/13)$, we ran 10000 simulations with 100000 vertices and counted the number which appeared to be converging to 1; simulations were estimated to be converging to $1$ if $x_{100000}>1/2$ and, to exclude cases of very slow convergence to 0, $x_{100000}>x_{90000}$.  The results (shown as numbers of simulations out of 10000) are in Table \ref{simul}.  It can be seen that convergence to $1$ is extremely rare for values of $\phi$ closer to 20/13, suggesting a continuous phase transition, but is reasonably common for values of $\phi$ a little larger than $1$.

\begin{table}[h]\begin{centre}\begin{tabular}{|c|ccccccccccc|}\hline $\phi$ & $1.00$ & $1.05$& $1.10$& $1.15$& $1.20$& $1.25$& $1.30$& $1.35$& $1.40$& $1.45$& $1.50$ \\ \hline Red domination & 3465 & 2425 & 1577 & 965 & 521 & 236 &  85 &  33  &  4  &  0  &  0\\ \hline \end{tabular}\caption{Numbers of simulations, out of 10000, appearing to show red domination for Section \ref{multdiscuss} Example \ref{biasblue}}\label{simul}\end{centre}\end{table}

\section*{Acknowledgement}

The author thanks two anonymous referees for useful comments and suggestions.

\bibliographystyle{plain}
\bibliography{scalefree}

\end{document}